\documentclass[11pt]{amsart}
\usepackage{amsmath,amssymb}
\usepackage{enumerate}
\usepackage{mathrsfs}
\usepackage{amsthm}
\usepackage{xcolor}

\newcommand\NN{\mathbb{N}}
\newcommand\RR{\mathbb{R}}

\newcommand\EE{\mathscr{E}}

\newcommand\diam{\mathrm{diam}}

\newcommand\pressure{\mathcal{P}}
\newcommand\Int{\mathrm{Int}}

\newcommand{\N}{\mathbb{N}}

\newcommand{\R}{\mathbb{R}}

\renewcommand{\Pi}{\pi}
\renewcommand{\emptyset}{\varnothing}

\newtheorem{theorem}{Theorem}[section]

\newtheorem{corollary}[theorem]{Corollary}

\newtheorem{example}[theorem]{Example}
\newtheorem{lemma}[theorem]{Lemma}
\newtheorem{remark}[theorem]{Remark}
\newtheorem{proposition}[theorem]{Proposition}

\title[Equality of H\"older exponents]{Equality of H\"older exponents for distribution functions of Gibbs measures}
\author{Pieter Allaart and Johannes Jaerisch}

\address{Mathematics Department, University of North Texas, 1155 Union Cir\#311430, Denton, TX 76203-5017, U.S.A.} 
\email{allaart@unt.edu}

\address{Graduate School of Mathematics, Nagoya University,
Furocho, Chikusaku, Nagoya, 464-8602, JAPAN} 
\email{jaerisch@math.nagoya-u.ac.jp}



\begin{document}

\begin{abstract}
Pointwise H\"older exponents describe the degree of regularity of a function near a point. For a function $f:\mathbb{R}\to\mathbb{R}$, a number $\alpha>0$ and a point $t_0\in\mathbb{R}$, write $f\in C^\alpha(t_0)$ if there exist a constant $C>0$, {a number $h>0$} and a polynomial $P$ of degree less than $\alpha$ such that
\[
|f(t)-P(t-t_0)|\leq C|t-t_0|^\alpha \qquad\mbox{for all {$t\in (t_0-h,t_0+h)$}}.
\]
The pointwise H\"older exponent of $f$ at $t_0$ is the number
\[
\alpha_f(t_0):=\sup\{\alpha>0: f\in C^\alpha(t_0)\}.
\]
A simpler quantity, also frequently called pointwise H\"older exponent in the mathematical literature, is the number
\[
\tilde{\alpha}_f(t_0):=\sup\{\alpha>0: f\in \tilde{C}^\alpha(t_0)\},
\]
where $f\in \tilde{C}^\alpha(t_0)$ means that there {exist $C>0$ and $h>0$} such that $|f(t)-f(t_0)|\leq C|t-t_0|^\alpha$ for all {$t\in (t_0-h,t_0+h)$}.
Clearly $\alpha_f(t)\geq \tilde{\alpha}_f(t)$, but strict inequality is possible and in fact common. In this paper we consider the case when $f=F_\mu$ is the distribution function of a Gibbs measure $\mu$ associated with an arbitrary H\"older continuous potential $\psi$ on a self-conformal set, and show that, under a very mild condition on $\psi$, $\alpha_f(t)=\tilde{\alpha}_f(t)$ for all $t$. As a consequence, we deduce that the pointwise H\"older spectrum of $f$ satisfies the multifractal formalism. As an application, we derive the pointwise H\"older spectrum of conjugacy maps between expanding piecewise $\mathcal{C}^{1+\epsilon}$ maps of an interval.
\end{abstract}

\maketitle

\section{Introduction}

A function $f:\RR\to\RR$ is said to be H\"older continuous with exponent $\alpha>0$ if there is a constant $C$ such that $|f(x)-f(y)|\leq C|x-y|^\alpha$ for all $x$ and $y$. However, this $\alpha$ represents the ``worst possible" behavior, and in general, a continuous function can at many points have substantially better regularity than the worst case. 

For $\alpha>0$ and $t_0\in\RR$, write $f\in C^\alpha(t_0)$ if there exist a constant $C>0$, {a number $h>0$} and a polynomial $P$ of degree less than $\alpha$ such that
\begin{equation}
|f(t)-P(t-t_0)|\leq C|t-t_0|^\alpha \qquad\mbox{for all {$t\in (t_0-h,t_0+h)$}}.
\label{eq:Holder-def}
\end{equation}
The {\em pointwise H\"older exponent} of $f$ at $t$ is the number
\begin{equation*}
\alpha_f(t):=\sup\{\alpha>0: f\in C^\alpha(t)\}, \qquad t\in\RR.
\end{equation*}
Of great interest are the level sets
\[
E_f(\alpha):=\{t\in\RR: \alpha_f(t)=\alpha\}, \qquad \alpha\geq 0.
\]
The {\em multifractal spectrum}, or {\em pointwise H\"older spectrum}, of $f$ is the function
\[
\alpha\mapsto \dim_H E_f(\alpha),
\]
where $\dim_H$ denotes Hausdorff dimension.

H\"older spectra are an important analytical tool in the study of certain physical processes that exhibit a wide range of local regularity behavior. They were introduced by Mandelbrot \cite{Mandelbrot} and Frisch and Parisi \cite{Frisch-Parisi} to study intermittent turbulence flows and intensity of seismic waves. Jaffard \cite{Jaffard1,Jaffard2} determined H\"older spectra for a class of self-similar functions satisfying a certain smoothness condition from a mathematical point of view using wavelet theory. Subsequent authors \cite{BKK,BenSlimane,JafMan,Seuret} computed H\"older spectra of specific self-affine functions; their results were generalized by the first author in \cite{Allaart-2018,Allaart-2020}. For pointwise H\"older spectra of Weierstrass-type functions, we refer to \cite{Otani}.

Pointwise H\"older exponents can be difficult to calculate, especially for $\alpha>2$, where, in order to show that $f\not\in C^\alpha(t_0)$, one must prove that no polynomial satisfying \eqref{eq:Holder-def} exists. For this reason perhaps, many authors use the following, simpler definition of H\"older exponent. Write $f\in \tilde{C}^\alpha(t_0)$ if there {exist $C>0$ and $h>0$} such that
\begin{equation*}
|f(t)-f(t_0)|\leq C|t-t_0|^\alpha \qquad\mbox{for all {$t\in(t_0-h,t_0+h)$}},
\end{equation*}
that is, the polynomial $P$ in \eqref{eq:Holder-def} is constant with value $f(t_0)$. Define
\begin{equation*}
\tilde{\alpha}_f(t):=\sup\{\alpha>0: f\in \tilde{C}^\alpha(t)\}, \qquad t\in\RR.
\end{equation*}
Clearly, $\alpha_f(t)\geq\tilde{\alpha}_f(t)$, but the reverse inequality may fail in general. For example, if $f(t)=t^2$, then $\alpha_f(t_0)=\infty$ for all $t_0\in\RR$ (since one can take $P(t)=(t+t_0)^2$ in \eqref{eq:Holder-def} for all $\alpha>2$), but $\tilde{\alpha}_f(0)=2$ whereas $\tilde{\alpha}_f(t_0)=1$ for all $t_0\neq 0$. In addition, a desirable property of H\"older exponents is that they are left unchanged upon perturbation of $f$ by a smooth function $g$. Indeed, $\alpha_f(t)=\alpha_{f+g}(t)$, whereas $\tilde{\alpha}_f(t)\neq\tilde{\alpha}_{f+g}(t)$ in general.

Unfortunately, the two definitions of pointwise H\"older exponent have frequently been conflated in the literature, often without proper justification. For this reason we feel that it is critically important to study classes of functions, ideally as general as possible, for which the two notions of H\"older exponent coincide. 

In general, if $f\in C^\alpha(t_0)$ for $\alpha>1$ and $f'(t_0)\neq 0$, then $\tilde{\alpha}_f(t_0)=1<\alpha_f(t_0)$. Turning this around, it follows that if $\alpha_f(t_0)=\tilde{\alpha}_f(t_0)$, then either $\alpha_f(t_0)\leq 1$ or $f'(t_0)=0$. This shows that the two definitions of H\"older exponent disagree for most functions known from calculus. On the other hand, many sufficiently irregular functions have the property that the only possible finite value of their derivative is zero.
For example, it was shown by the first author in \cite{Allaart-2018,Allaart-2020} that for certain classes of self-affine functions, including distribution functions of self-similar measures on $\RR$ satisfying the open set condition, the two definitions of pointwise H\"older exponent coincide.
In this note, we extend this result to distribution functions of arbitrary Gibbs measures on finitely generated self-conformal sets; see Theorem \ref{thm:main} for a precise statement. Although it may seem plausible that the result can be extended from self-similar measures to self-conformal ones, the proof nonetheless requires an entirely new approach. This is because the method of proof in \cite{Allaart-2020}, using divided differences, depends very explicitly on the strict self-similarity of the measure and can therefore not be extended to the self-conformal setting.

If $f=F_\mu$ is the distribution function of a Borel probability measure $\mu$, i.e. $F_{\mu}(x):=\mu\left((-\infty,x]\right)$ for $x\in\RR$,
then by \cite[Lemma 5.1]{JS-2020} we have for each $t\in \R$, \begin{equation}
\tilde{\alpha}_{F_\mu}(t)=\liminf_{r\rightarrow 0}\frac{\log\mu\left(B(t,r)\right)}{\log r}.
\label{eq:hoelder exponent as liminf}
\end{equation}
For the remainder of this introduction, let $\mu$ be a Gibbs measure associated with an arbitrary H\"older continuous potential on the attractor  of a conformal iterated function system (IFS) satisfying the open set condition -- see the next section for precise definitions. It is well known that the sets
\[
\EE(\alpha;\mu):=\left\{ x\in\RR\Big|\lim_{r\rightarrow0}\frac{\log\mu\left(B(x,r)\right)}{\log r}=\alpha\right\}, \qquad\alpha\in\R,
\]
satisfy the {\em multifractal formalism}; that is, $\dim_H \EE(\alpha;\mu)=\dim_P \EE(\alpha;\mu)=\beta^*(\alpha)$, where $\beta^*(\alpha)$ is the Legendre transform of a suitable scaling function $\beta(q)$, defined in the next section. This result was first proved in \cite{Patzschke}  for conformal iterated functions, and in \cite{Pesin-Weiss-97a} for conformal repellers (see also \cite{pesin-dimension-theory}  for a detailed account and related results).
The second author and Sumi \cite{JS-2020} showed recently that, in the case when all of the maps in the IFS are strictly increasing, the Hausdorff spectrum of the modified level sets
\[
\EE_*(\alpha;\mu):=\left\{ x\in\RR\Big|\liminf_{r\rightarrow0}\frac{\log\mu\left(B(x,r)\right)}{\log r}=\alpha\right\}, \qquad\alpha\in\R,
\]
is also given by the multifractal formalism: $\dim_H \EE_*(\alpha;\mu)=\beta^*(\alpha)$. 
It is straightforward to extend this result to general conformal IFS. Thus, as a consequence of our main theorem, the pointwise H\"older spectrum of $F_\mu$ is given by
\[
\dim_H E_{F_\mu}(\alpha)=\beta^*(\alpha).
\]



As an application of our main theorem, we determine the pointwise H\"older spectrum of conjugacy maps between expanding piecewise $\mathcal{C}^{1+\epsilon}$
interval maps in Corollary \ref{cor}. This result complements  \cite[Corollary 2.2]{JS-2020}, where the dimension spectrum of the conjugacy map is based on the simpler definition of the H\"older exponent $\tilde \alpha$. We note that it is straightforward to extend Corollary \ref{cor} to the framework of expanding circle endomorphisms. For related results on the non-differentiability of conjugacy maps in both frameworks we refer to \cite{JKPS-2009}.

\begin{remark}
{\rm
A closely related concept is that of $\alpha$-H\"older differentiability. For $\alpha>0$, say a function $f$ is {\em $\alpha$-H\"older differentiable} at a point $t_0$ if the limit
\[
\lim_{t\to t_0}\frac{|f(t)-f(t_0)|}{|t-t_0|^\alpha}
\]
exists. The existence of the above limit ($0$, positive and finite, or $+\infty$) gives information about the H\"older exponent $\tilde{\alpha}_f(t_0)$ (but not, in general, about $\alpha_f(t_0)$), and vice versa. H\"older differentiability for distribution functions of Gibbs measures was considered in \cite{Falconer-2004, KS-2009}, with a later improvement in \cite{Troscheit}. Although some of the methods used in these papers (including the use of pressure functions and the bounded distortion principle) are similar to those used here, the principal questions are fundamentally different.
}
\end{remark}

This paper is organized as follows: In Section \ref{sec:notation}, we introduce the relevant notation and definitions. In Section \ref{sec:main-theorem}, we state and prove our main result, namely the equality of $\alpha_F(t)$ and $\tilde{\alpha}_F(t)$ when $F$ is the distribution function of a Gibbs measure associated with a H\"older continuous potential $\psi$ on a self-conformal attractor. The only condition we impose on $\psi$ is that it is not cohomologous to the geometric potential $\varphi$ of the IFS; this is to avoid cases where $F$ is absolutely continuous and hence too ``regular". The technical heart of the proof is Proposition \ref{prop:no-positive-derivative}, which states that  for each positive integer $k$ the limit
\[
\lim_{y\to x}\frac{F(y)-F(x)}{(y-x)^k},
\]
if it exists, cannot be nonzero and finite.

As a consequence of the main result, we derive in Section \ref{sec:Holder-spectrum} that the (Hausdorff) multifractal spectrum of such a distribution function is given by the multifractal formalism. In this section we show also that, at least in the self-similar case, the packing multifractal spectrum does {\em not} satisfy the multifractal formalism. Finally, in Section \ref{sec:conjugacy-maps}, we give an application of our main theorem to conjugacy maps between expanding piecewise $\mathcal{C}^{1+\epsilon}$ maps.

\section{Notation and definitions} \label{sec:notation}

Let  $X\subset \RR$ be a compact interval of positive length. A $\mathcal{C}^{1+\epsilon}$ conformal iterated function system on $X$ is given by an index set $I$ with $2\le \#I<\infty$ and $\mathcal{C}^{1+\epsilon}$ conformal contracting maps $\Phi=(\phi_{i})_{i\in I}$, $\phi_{i}:X\rightarrow X$.
More precisely, we mean that,   for each $i\in I$, the map $\phi_i$ extends to a 
$\mathcal{C}^{1+\epsilon}$ differentiable map on a neighborhood of $X$ which 
satisfies $0<\left| \phi_{i}'(x) \right|<1$ for every $x\in X$. 
We refer to \cite{Mauldin-Urbanski} for further details and basic properties of conformal 
iterated function systems.
Let $\pi:I^{\N}\rightarrow\R$ denote the coding map of $\Phi$ which is for $\omega=(\omega_{1}\omega_{2}\dots)\in I^{\N}$ given by 
\[
\bigcap_{n\ge1}\phi_{\omega_{1}}\circ\dots\circ\phi_{\omega_{n}}(X)=\left\{ \pi(\omega)\right\} .
\]
Let $\Lambda:=\pi(I^{\N})$; then $\Lambda$ satisfies the set equation $\Lambda=\bigcup_{i\in I}\phi_i(\Lambda)$. We refer to $\Lambda$ as the \emph{self-conformal set} generated by $\Phi$. We say that $\Phi$ satisfies the {\em open set condition} if $\phi_{i}(\Int(X))\cap\phi_{j}(\Int(X))=\emptyset$ for all $i,j\in I$ with $i\neq j$.
To utilize the symbolic thermodynamic formalism (see \cite{bowen-equilibrium}) we will also need the following definitions. We denote by $\sigma:I^{\N}\rightarrow I^{\N}$ the left-shift on $I^{\N}$ which becomes a compact metric space endowed with the shift metric. A basis for the topology is given by the cylinder sets 
\[
[w_1\dots w_n]:=\{ \omega \in I^\N:\omega_1=w_1,\dots, \omega_n=w_n\}, \qquad w_1\dots w_n\in I^n, \quad n\in\NN.
\]
Let $\varphi:I^{\N}\rightarrow\R$ be the geometric
potential of $\Phi$, given by
\begin{equation} \label{eq:geometric-potential}
\varphi(\omega):=\log\left|\phi_{\omega_{1}}'(\pi(\sigma(\omega)))\right|,\quad\omega=\omega_{1}\omega_{2}\dots\in I^{\N}.
\end{equation}
Since each $\phi_i$ is $\mathcal{C}^{1+\epsilon}$ differentiable and the maps $\phi_i$ are uniformly contracting, it follows that $\varphi$ is H\"older continuous. Let $\psi:I^{\N}\rightarrow\R$ be H\"older continuous and denote by 
\[
S_{n}\psi:=\sum_{k=0}^{n-1}\psi\circ\sigma^{k}, \qquad n\geq 1
\]
its ergodic sums. We will frequently make use of the fact that any H\"older continuous $\psi:I^\N\rightarrow \R$ satisfies the {\em bounded distortion property}
\begin{equation}\label{eq:bdd}
\sup_{n\in \N} \sup_{w\in I^n} \sup_{\rho,\tau\in I^\N }\left|S_n \psi(w\rho)-S_n\psi(w\tau)\right|<\infty.
\end{equation} We say that a Borel probability measure $\mu_\psi$ on $I^{\N}$  is a Gibbs measure for $\psi$  if there exist two constants $C\ge 1$ and $P\in \R$ such that for each $n\in \N$ and for all $w \in I^n$ we have 
\begin{equation}\label{eq:gibbs}
C^{-1}\le \frac{\mu_\psi[w_1\dots w_n]}{e^{S_n\psi(\overline w)-nP }}\le C.
\end{equation} Here $\overline{w}$ denotes the infinite periodic sequence with period block $w$. 
It is well known that a Gibbs measure $\mu_\psi$ for a H\"older continuous potential always exists  \cite{bowen-equilibrium}.
Further note that the constant $P$ in the definition of a Gibbs measure is equal to  the topological pressure of $\psi$ with respect to $\sigma$, defined by
\[
\mathcal{P}(\psi):=\lim_{k\to\infty}\frac{1}{k}\log\sum_{w\in I^k}\exp S_k \psi(\overline{w}).
\]

Two H\"older potentials $\psi$ and $\xi$ on $I^\NN$ are said to be {\em cohomologous} if there is a bounded potential $\eta:I^\NN\to\RR$ such that $\psi-\xi=\eta-\eta\circ\sigma$. It is easy to see that, if $\psi$ and $\xi$ are cohomologous, then $\mathcal{P}(\psi)=\mathcal{P}(\xi)$.

We remark that if $\mu_\psi$ is a Gibbs measure for $\psi$, then it is also a Gibbs measure for $\psi-\pressure(\psi)$. We may therefore always assume that $\pressure(\psi)=0$ by replacing $\psi$ with $\psi-\pressure(\psi)$. 

Let 
\begin{equation} \label{spectrum-endpoints}
\alpha_{+}:=\sup_{\omega \in I^\N}\limsup_{n\rightarrow \infty } \frac{S_n \psi(\omega)}{S_n \varphi (\omega)} 
, \qquad
\alpha_{-}:=\inf_{\omega \in I^\N}\liminf_{n\rightarrow \infty } \frac{S_n \psi(\omega)}{S_n \varphi (\omega)} .
\end{equation}
It is well known that the multifractal spectrum is complete \cite{schmeling1999}, that is, $\EE(\alpha;\mu_\psi)\neq\emptyset$ if and only if $\alpha\in\left[\alpha_{-},\alpha_{+}\right]$. 
Moreover, the following dichotomy is well known \cite{Pesin-Weiss-97a}: either we have $\alpha_{-}<\alpha_{+}$, or $\psi$ is cohomologous to $(\dim_{H}\Lambda)\varphi$. Note that in the latter case we have  $\alpha_{-}=\alpha_{+}=\dim_{H}\Lambda$.

Since $\varphi<0$, there is for each $q\in\R$ a unique $\beta(q)\in\RR$ such that 
\begin{equation}
\mathcal{P}(\beta(q)\varphi+q\psi)=0. \label{eq:pressure equation}
\end{equation}
We denote by 
\begin{equation} \label{eq:Legendre-transform}
\beta^{*}(\alpha):=\inf_{q\in\R}\left\{\beta(q)+\alpha q\right\}
\end{equation}
the Legendre transform of $\beta$. 
If $\alpha_-<\alpha_+$, then it is well known that $\beta^*$ is real-analytic and strictly concave on $(\alpha_-,\alpha_+)$ \cite{pesin-dimension-theory}.

\section{The main theorem} \label{sec:main-theorem}

Our main result is the following:

\begin{theorem} \label{thm:main}
Let $\Phi=(\phi_{i}:X\rightarrow X)_{i\in I}$ be a conformal iterated
function system satisfying the open set condition. Let $\pi:I^\NN\to\RR$ denote the coding map of $\Phi$. 
Let $\varphi$ denote the geometric potential of $\Phi$ defined in \eqref{eq:geometric-potential}  and let $\psi$ be any H\"older potential with $\pressure(\psi)=0$ such that $\psi$ is not cohomologous to $\varphi$.  Let $\mu_\psi$ be a Gibbs measure for $\psi$. Then, for the distribution function $F$ of $\mu_\psi\circ\pi^{-1}$, we have that $\alpha_F=\tilde{\alpha}_F$. 
\end{theorem}

\begin{remark}
{\rm
The condition that $\psi$ is not cohomologous to $\varphi$ is needed to rule out cases where $F$ is too ``smooth".
As a simple example, consider the IFS $\Phi=\{x/2,(x+1)/2\}$, whose attractor is the interval $[0,1]$. Clearly, $\varphi\equiv -\log 2$. Choosing $\psi\equiv -\log 2$ as well, we see that $\mu:=\mu_\psi\circ\pi^{-1}$ is Lebesgue measure on $[0,1]$, so $F$ is infinitely differentiable on $(0,1)$ and hence $\alpha_F(t)\equiv \infty$ on $(0,1)$. On the other hand, it is easy to see that $\tilde{\alpha}_F(t)\equiv 1$ on $(0,1)$.
}
\end{remark}

More generally, we have the following characterization:

\begin{proposition} \label{prop:when-cohomologous}
 We have that $\psi$ is cohomologous to $\varphi$ if and only if $\Lambda$ is an interval and some (hence any) Gibbs measure $\mu_\psi$ for $\psi$ satisfies that $\mu_\psi \circ \pi^{-1}$ is  absolutely continuous with respect to the Lebesgue measure with a Radon-Nykodym density bounded away from zero and infinity.
\end{proposition}

\begin{proof} 
First assume that $\psi$ is cohomologous to $\varphi$. Hence, $P(\varphi)=P(\psi)=0$
and thus, by Bowen's formula (see e.g. \cite[Theorem 4.2.11]{MauUrb03}),
$\Lambda$ has positive Lebesgue measure. We conclude that $X=\bigcup_{i\in I}\Phi_{i}(X)$
and therefore, $\Lambda=X$ is an interval. Indeed, if $X\setminus\bigcup_{i\in I}\Phi_{i}(X)$
was non-empty, it would contain an interval of positive length. By
standard arguments, this would imply that $\Lambda$ contains no Lebesgue
density points, which gives the desired contradiction. Let $\lambda$
denote the normalised Lebesgue measure on $X$. Since $\Lambda=X$
and $P(\varphi)=0$, the unique Borel probability measure $\mu$ on
$I^{\mathbb{N}}$ satisfying $\mu\circ\pi^{-1}=\lambda$ is a Gibbs
measure for $\varphi$. Since $\psi$ is cohomologous to $\varphi$,
it follows from \cite[Theorem 2.2.7]{MauUrb03} that any Gibbs measure
$\mu_{\psi}$ for $\psi$ is equivalent to $\mu$ with a density bounded
away from zero and infinity. Hence, $\mu_{\psi}\circ\pi^{-1}$ is
equivalent to $\lambda$ with a density bounded away from zero
and infinity. 

Conversely, assume that $\Lambda$ is an interval. Then, $\Lambda=X$
and we have $P(\varphi)=0$ by Bowen's formula. Further, the unique
Borel probability measure $\mu$ on $I^{\mathbb{N}}$ defined by $\mu\circ\pi^{-1}=\lambda,$
is a Gibbs measure for $\varphi$. Assume that $\mu_{\psi}$ is a
Gibbs measure for $\psi$ satisfying that $\mu_{\psi}\circ\pi^{-1}$
is absolutely continuous with respect to $\lambda$ with density
bounded away from zero and infinity. Then, $\mu$ is also a Gibbs
measure for $\psi$. Since $P(\varphi)=P(\psi$),  it follows from
\cite[Theorem 2.2.7]{MauUrb03} that $\varphi$ is cohomologous to
$\psi$. 
\end{proof}


The remainder of this section is devoted to the proof of Theorem \ref{thm:main}.
{Let $N:=|I|$ denote the cardinality of $I$.} We will assume without loss of generality that the images of $X$ under the maps $\phi_1,\dots,\phi_N$ are arranged from left to right; that is, $\inf\phi_{i+1}(X)\geq \sup\phi_i(X)$ for $i=1,\dots,N-1$. The first lemma is needed to deal with the codings of points that lie close to the common endpoint of adjacent cylinder intervals. The precise conclusion depends on the orientations of the first and last maps, $\phi_1$ and $\phi_N$, of the IFS. In this section, we write $\mathbf{i}|_n:=i_1\dots i_n$ for $\mathbf{i}=i_1\dots i_p\in I^p$ with $p\geq n$; and similarly, $\omega|_n:=\omega_1\dots\omega_n$ for $\omega=\omega_1\omega_2\dots\in I^\NN$. {For a word $u\in I^*$ and $k\in\NN\cup\{\infty\}$, $u^k$ will denote the $k$-fold concatenation of $u$ with itself.}

\begin{lemma} \label{lem:adjacency-possibilities}
Let $\mathbf{i}=i_1\dots i_p, \mathbf{j}=j_1\dots j_p\in I^p$ with $\mathbf{i}\neq\mathbf{j}$, where $p\geq 3$. Let $n:=|\mathbf{i}\wedge\mathbf{j}|+1$ be the first coordinate at which $\mathbf{i}$ and $\mathbf{j}$ disagree. Assume $n\leq p-2$, and suppose the cylinder intervals $\pi([\mathbf{i}])$ and $\pi([\mathbf{j}])$ have a common endpoint. Then $|i_n-j_n|=1$ and one of the following holds:
\begin{enumerate}[(i)]
\item $i_{n+1}\dots i_p\in\big\{1^{p-n},N^{p-n},1N^{p-n-1},N1^{p-n-1}\big\}$; or
\item $i_{n+1}\dots i_p$ is a prefix of $(1N)^\infty$ and $i_n\neq N$; or
\item $i_{n+1}\dots i_p$ is a prefix of $(N1)^\infty$ and $i_n\neq 1$.
\end{enumerate}
\end{lemma}

{
\begin{remark}
{\rm 
Although we stated the conclusion for the word $\mathbf{i}$, by the symmetry of the hypotheses the same conclusion holds also for the word $\mathbf{j}$.
}
\end{remark}
}

\begin{proof}
Since the cylinder intervals $\pi([1]),\dots,\pi([N])$ are arranged from left to right, the cylinder intervals $\pi([\mathbf{i}|_{n-1}1]),\dots,\pi([\mathbf{i}|_{n-1}N])$ are arranged either from left to right or from right to left in $\pi([\mathbf{i}|_{n-1}])$ (depending on the orientation of $\phi_{i_1}\circ\dots\circ\phi_{i_{n-1}}$). Hence, $|i_n-j_n|=1$.

Assume without loss of generality that $\pi([\mathbf{i}|_n])$ lies to the left of $\pi([\mathbf{j}|_n])$. Then for each $l=n+1,\dots,p$, $\pi([\mathbf{i}|_l])$ is the rightmost subcylinder of $\pi([\mathbf{i}|_{l-1}])$ and $\pi([\mathbf{j}|_l])$ is the leftmost subcylinder of $\pi([\mathbf{j}|_{l-1}])$.
Thus, $i_l,j_l\in\{1,N\}$ for $l=n+1,\dots,p$.

If $\phi_1$ is increasing, then in $i_{n+1}\dots i_p$ any digit $1$ must be succeeded by another $1$. Similarly, if $\phi_N$ is increasing, any digit $N$ must be followed by another $N$. On the other hand, if $\phi_1$ is decreasing, any digit $1$ must be followed by a digit $N$; and likewise, if $\phi_N$ is decreasing, any digit $N$ must be followed by a digit $1$. This implies that $i_{n+1}\dots i_p$ is of one of the forms in (i) or else is a prefix of either $(1N)^\infty$ or $(N1)^\infty$. 

Suppose $i_{n+1}\dots i_p$ is a prefix of $(1N)^\infty$. Then $\phi_1$ and $\phi_N$ are both decreasing. If $i_n=N$, then necessarily $j_n=N-1$. Since $\pi([\mathbf{i}|_n])$ lies to the left of $\pi([\mathbf{j}|_n])$, this means that the map $\phi_{i_1}\circ\dots\circ\phi_{i_{n-1}}$ is decreasing. But $\phi_N$ is also decreasing, and so $\phi_{i_1}\circ\dots\circ\phi_{i_n}$ is increasing. Therefore, the rightmost subcylinder of $\pi([\mathbf{i}|_{n}])$ is $\pi([\mathbf{i}|_n N])$, contradicting that $i_{n+1}=1$. Hence, $i_n\neq N$. A similar argument shows that if $i_{n+1}\dots i_n$ is a prefix of $(N1)^\infty$, then $i_{n+1}\neq 1$.
\end{proof}

\begin{lemma} \label{lem:non-adjacency}
Let {$\ell\geq 3$} and let $\tau_1\dots\tau_\ell\in I^\ell$ be a word such that $\tau_1\neq\tau_2$. Then for any $\omega\in I^\NN$ there are infinitely many integers $n$ such that the cylinder intervals $\pi([\omega|_{n+3\ell}])$ and $\pi([\omega|_{n}(\tau_1\dots\tau_\ell)^3])$ do not have a common endpoint.
\end{lemma}

\begin{proof}
Let $\omega\in I^\NN$, and consider two cases:

{\em Case 1.} Suppose first that there exist infinitely many $n$ such that $\omega_{n}\neq\tau_1$. Take such an $n$. Suppose, by way of contradiction, that the cylinder intervals $\pi([\omega|_{n-1}\omega_{n}\dots\omega_{n+\ell+1}])$ and $\pi([\omega|_{n-1}\tau_1\dots\tau_\ell\tau_1\tau_2])$ are adjacent. Then either (i), (ii) or (iii) of Lemma \ref{lem:adjacency-possibilities} holds for the word $\mathbf{i}:=\omega|_{n-1}\tau_1\dots\tau_\ell\tau_1\tau_2$, with $p:=n+\ell+1$, so that $i_{n+1}\dots i_p=\tau_2\dots\tau_\ell\tau_1\tau_2$. Since $\tau_2\neq\tau_1$, we can rule out (i). If $\tau_2\dots\tau_n\tau_1\tau_2$ is a prefix of $(1N)^\infty$, then $\tau_2=1$ and so $\tau_1=N$, which means (ii) does not hold. By a symmetric argument, (iii) does not hold. Hence, $\pi([\omega|_{n+\ell+1}])$ and $\pi([\omega|_{n-1}\tau_1\dots\tau_\ell\tau_1\tau_2])$ are not adjacent, so neither are $\pi([\omega|_{n-1+3\ell}])$ and $\pi([\omega|_{n-1}(\tau_1\dots\tau_\ell)^3])$, since $\ell+2\leq 3\ell$.

{\em Case 2.} Suppose $\omega$ ends in $\tau_1^\infty$. Then $\omega_{n-1}=\tau_1$ and $\omega_n\neq\tau_2$ for all large enough $n$. Fix such a large $n$. Recall that $\ell\geq 3$. Apply Lemma \ref{lem:adjacency-possibilities} with $p:=n+\ell+1$, $\mathbf{i}:=\omega|_{n-1}\tau_2\dots\tau_\ell\tau_1\tau_2\tau_3=\omega|_{n-2}\tau_1\dots\tau_\ell\tau_1\tau_2\tau_3$ and $\mathbf{j}:=\omega|_{n+\ell+1}$, so $i_{n+1}\dots i_p=\tau_3\dots\tau_\ell\tau_1\tau_2\tau_3$. We can immediately rule out condition (i) of Lemma \ref{lem:adjacency-possibilities}, because $i_{n+1}=i_p$ and $i_{p-2}=\tau_1\neq\tau_2=i_{p-1}$. On the other hand, if $\tau_3\dots\tau_n\tau_1\tau_2\tau_3$ is a prefix of $(1N)^\infty$, then $\tau_3=1$ and hence $i_n=\tau_2=N$, so (ii) does not hold. Similarly, (iii) does not hold. Thus, the cylinder intervals $\pi([\mathbf{i}])$ and $\pi([\mathbf{j}])$ are not adjacent, so neither are $\pi([\omega|_{n-2+3\ell}])$ and $\pi([\omega|_{n-2}(\tau_1\dots\tau_\ell)^3])$, since $\ell+3\leq 3\ell$.

\end{proof}


\begin{lemma} \label{lem:cohom-1} 
Suppose there exists $\ell_0\ge 1$ such that the set
\[
A:=\left\{\frac{S_{\ell}\psi\left(\overline{\tau_{1}\dots \tau_{\ell}}\right)}{S_{\ell}\varphi\left(\overline{\tau_{1}\dots \tau_{\ell}}\right)}:\quad \ell \ge \ell_0,\quad \tau_1, \dots, \tau_\ell \in I \right\}
\]
is finite.
Then $\psi$ is cohomologous to $(\dim_{H}\Lambda)\varphi$  and $ A=\left\{\dim_{H}\Lambda\right\} $.
\end{lemma}

\begin{proof}
By the bounded distortion property \eqref{eq:bdd}, it is easy to see that for each
$\ell_{0}\ge1$,
\[
[\alpha_{-},\alpha_{+}]=\overline{\left\{ \frac{S_{\ell}\psi\left(\overline{\tau_{1}\dots \tau_{\ell}}\right)}{S_{\ell}\varphi\left(\overline{\tau_{1}\dots \tau_{\ell}}\right)}:\ell\ge\ell_{0},\tau_{1},\dots,\tau_{\ell}\in I\right\} }.
\]
Since $A$ is finite, we have $\alpha_-=\alpha_+$. By a well-known dichotomy for multifractal spectra  \cite{Pesin-Weiss-97a}, {it follows} that $\psi$ is cohomologous to $(\dim_{H}\Lambda)\varphi$. Clearly, this implies $A=\left\{\dim_{H}\Lambda\right\}$.
\end{proof}

\begin{proposition} \label{prop:no-positive-derivative}
For every $k\in\mathbb{N}$ and $x\in\mathbb{R}$ the limit 
\[
\lim_{y\rightarrow x}\frac{F(y)-F(x)}{(y-x)^{k}}
\]
either does not exist, or it takes the value $0$ or $+\infty$. 
\end{proposition}

\begin{proof}
{If $x\in\RR\backslash\Lambda$, then $F$ is constant on a neighborhood of $x$ since $\Lambda$ is closed, and hence the limit in the proposition is equal to $0$. Therefore, we may assume that $x\in\Lambda$.}
The statement is trivial for even $k$, since $F$ is increasing. So below we will assume that $k$ is odd.

Suppose by way of contradiction that 
\[
\lim_{y\rightarrow x}\frac{F(y)-F(x)}{(y-x)^{k}}=L\in(0,+\infty).
\]
Note that the limit cannot be negative, since $F$ is increasing.

Let $\omega \in I^\N$ such that $\pi(\omega)=x$.
Let $[s_n,t_n]:=\pi\left(\left[\omega|_{n}\right]\right)$. Then $s_n\leq x\leq t_n$, and
\begin{align}
\begin{split}
\frac{F(t_{n})-F(s_{n})}{(t_{n}-s_{n})^{k}} & =\left(\frac{t_{n}-x}{t_{n}-s_{n}}\right)^{k}\cdot\frac{F(t_{n})-F(x)}{(t_{n}-x)^{k}}\\
 & \quad+\left(\frac{x-s_{n}}{t_{n}-s_{n}}\right)^{k}\cdot\frac{F(s_{n})-F(x)}{(s_{n}-x)^{k}}.
\end{split}
\label{eq:secant-slope}
\end{align}
Denote 
$$r_{n}:=\frac{t_{n}-x}{t_{n}-s_{n}}.$$ 
Observe that $r_n\in[0,1]$ and
$$1-r_{n}=\frac{x-s_{n}}{t_{n}-s_{n}}.$$

There exists  $\ell\in\mathbb{N}$ and $\tau_{1}\dots \tau_{\ell}\in I^\ell$
containing at least two different letters such that 
\[
S_{\ell}(\psi-k\varphi)\left(\overline{\tau_{1}\dots \tau_{\ell}}\right)\neq0.
\]
Indeed such $\tau_{1}\dots \tau_{\ell}$ always exist for otherwise, with the set  $A$ in Lemma \ref{lem:cohom-1} given by the finite set $\{k\}\cup\left\{ \psi(\overline{i})/\varphi(\overline{i}):i\in I\right\}$, Lemma \ref{lem:cohom-1} implies that $\psi$ is cohomologous to $(\dim_{H}\Lambda)\varphi$
and $\dim_{H}\Lambda=k=1$ {(since $\Lambda\subset\RR$)}, contradicting our hypothesis. Replacing $\tau_1\dots\tau_\ell$ by a cyclical permutation if necessary, we may assume that $\tau_2\neq\tau_1$. {Further, we may assume without loss of generality that $\ell\geq 3$: If $\ell=2$, we can replace $\ell$ and $\tau_1\tau_2$ with $\ell'=4$ and $\tau_1'\dots\tau_4'=(\tau_1\tau_2)^2$, and achieve the same effect.}

Taking a subsequence if necessary, we may assume that $\lim r_{n}=r\in[0,1]$.
Then by \eqref{eq:secant-slope}, we have
\[
\frac{F(t_{n})-F(s_{n})}{(t_{n}-s_{n})^{k}}\rightarrow\left(r^{k}+(1-r)^{k}\right)L\in\left[2^{1-k}L,L\right],
\]
as routine calculus shows that the function $r\mapsto r^k+(1-r)^k$ is minimized at $r=1/2$.
Hence,
\begin{equation} \label{eq:bounded-ratio}
\lim_{n\rightarrow\infty}\frac{\mu_\psi\left(\left[\omega|_{n}\right]\right)}{\diam\left(\pi\left(\left[\omega|_{n}\right]\right)\right)^{k}}=\lim_{n\rightarrow\infty}\frac{F(t_{n})-F(s_{n})}{(t_{n}-s_{n})^{k}}\in\left[2^{1-k}L,L\right]. 
\end{equation}

Now, we fix an integer {$m\geq 3$} and consider the interval
\[
\pi\left(\left[\omega|_{n}(\tau_{1}\dots \tau_{\ell})^{m}\right]\right)=:\left[s_{n,m},t_{n,m}\right].
\]
Put
\[
r_{n,m}:=\frac{t_{n,m}-x}{t_{n,m}-s_{n,m}}.
\]
By Lemma \ref{lem:non-adjacency} there is an infinite subset $\mathcal{N}\subset\NN$ such that for each $n\in\mathcal{N}$, the point $x$ and the interval $[s_{n,m},t_{n,m}]$ are separated by a full cylinder interval of level $n+3\ell$. 
Replacing $\mathcal{N}$ by an infinite subset if necessary, we may assume that either $t_{n,m}>x$ for all $n\in \mathcal{N}$, or $t_{n,m}<x$ for all $n\in \mathcal{N}$. Without loss of generality, we assume the former. The argument in the latter case is similar. Thus, $r_{n,m}>0$ for all $n\in\mathcal{N}$.

Let $r_{\min}>0$ be a number such that $|\phi_i'(x)|\geq r_{\min}$ for all $i\in I$ and $x\in X$. Then
\begin{equation} \label{eq:diameters}
\diam (\pi([uv]))\geq \diam (\pi([u]))\cdot r_{\min}^{|v|}
\end{equation}
for any two finite words $u$ and $v$.

For $n\in\mathcal{N}$, let $[u_n,v_n]$ be a cylinder interval of length $n+3\ell$ separating $x$ and $[s_{n,m},t_{n,m}]$. Then by \eqref{eq:diameters},
\[
|t_{n,m}-x|\geq v_n-u_n\geq \diam (\pi([\omega|_n]))\cdot r_{\min}^{3\ell}.
\]
Hence, there is a constant $C=C(\ell)>1$, depending only on $\ell$ but not on $n$ or $m$, such that
\[
C^{-1}<\frac{t_{n,m}-x}{\diam\big(\pi\left(\left[\omega|_{n}\right]\right)\big)}<C,
\]
for all $n\in\mathcal{N}$. We denote this by
\begin{equation} \label{eq:comparable}
t_{n,m}-x\asymp_\ell\diam\big(\pi\left(\left[\omega|_{n}\right]\right)\big). 
\end{equation}

Now \eqref{eq:comparable} implies 
\begin{equation}
\label{eq:rn-bdd}
r_{n,m}=\frac{t_{n,m}-x}{t_{n,m}-s_{n,m}}\asymp_\ell\frac{\diam\left(\pi\left(\left[\omega|_{n}\right]\right)\right)}{\diam\big(\pi\left(\left[\omega|_{n}(\tau_{1}\dots \tau_{\ell})^{m}\right]\right)\big)}\asymp_\ell e^{-mS_{\ell}\left(\varphi\right)(\overline{\tau_{1}\dots \tau_{\ell}})}.
\end{equation}
Hence, by passing to a subsequence, using a diagonal argument, we may assume that $r_{n,m}$ converges to some number $q_m \in \mathbb{R}$ for every $m\ge 1$.  {Since $r_{n,m}>0$ for every $n\in\mathcal{N}$, it follows that $q_m\geq 0$.}

Using {the analogue of \eqref{eq:secant-slope} with $s_{n,m}$ and $t_{n,m}$ in place of $s_n$ and $t_n$} and letting $n\rightarrow \infty$ we obtain 
\begin{equation} \label{eq:limit-bounded-below}
\lim_{n\to\infty}\frac{F(t_{n,m})-F(s_{n,m})}{(t_{n,m}-s_{n,m})^k} =\left(q_{m}^k+(1-q_{m})^k\right)L \ge 2^{1-k}L.
\end{equation}

On the other hand, by the Gibbs property \eqref{eq:gibbs} of $\mu_\psi$, the bounded distortion property \eqref{eq:bdd} and \eqref{eq:bounded-ratio} we have  
\begin{align*}
\frac{F(t_{n,m})-F(s_{n,m})}{(t_{n,m}-s_{n,m})^k} & =\frac{\mu_\psi\left(\left[\omega|_{n}(\tau_{1}\dots \tau_{\ell})^{m}\right]\right)}{\left(\diam(\pi\left(\left[\omega|_{n}(\tau_{1}\dots \tau_{\ell})^{m}\right])\right)\right)^k}\\
 & \asymp_\ell\frac{\mu_\psi\left(\left[\omega|_{n}\right]\right)}{\left(\diam(\pi\left(\left[\omega|_{n}\right])\right)\right)^k}\cdot e^{mS_{\ell}\left(\psi-k\varphi\right)(\overline{\tau_{1}\dots \tau_{\ell}})}\\
 & \asymp_\ell e^{mS_{\ell}\left(\psi-k\varphi\right)(\overline{\tau_{1}\dots \tau_{\ell}})}.
\end{align*}
If $k=1$ this already gives a contradiction, since we then have equality in \eqref{eq:limit-bounded-below}, which implies $S_{\ell}\left(\psi-k\varphi\right)(\overline{\tau_{1}\dots \tau_{\ell}})=0$.

Suppose now that $k\geq 3$. We have shown that
\[
2^{1-k}L\le \big(q_{m}^k+(1-q_{m})^k\big)L= \lim_{n\to\infty}\frac{F(t_{n,m})-F(s_{n,m})}{(t_{n,m}-s_{n,m})^k} \asymp_{\ell} e^{mS_{\ell}\left(\psi-k\varphi\right)(\overline{\tau_{1}\dots \tau_{\ell}})}.
\]
Since $S_{\ell}(\psi-k\varphi)\left(\overline{\tau_{1}\dots \tau_{\ell}}\right)\neq0$
it follows that $S_{\ell}(\psi-k\varphi)\left(\overline{\tau_{1}\dots \tau_{\ell}}\right)>0$, and hence, $q_{m}\rightarrow\infty$ as $m\rightarrow\infty$.
Since $k$ is odd, $r^k+(1-r)^k=kr^{k-1}+O(r^{k-2})$ as $r\to\infty$, so we conclude that $q_m^k+(1-q_m)^k\asymp q_m^{k-1}$. Thus,
\[
q_{m}^{k-1}\asymp e^{mS_{\ell}\left(\psi-k\varphi\right)(\overline{\tau_{1}\dots \tau_{\ell}})}.
\]
Together with \eqref{eq:rn-bdd}, this implies
\[
e^{mS_{\ell}\left(\psi-k\varphi\right)(\overline{\tau_{1}\dots \tau_{\ell}})} \asymp q_m^{k-1} \asymp e^{-m(k-1)S_{\ell}\left(\varphi\right)(\overline{\tau_{1}\dots \tau_{\ell}})},
\]
and hence, 
\[
S_{\ell}\left(\psi-k\varphi\right)(\overline{\tau_{1}\dots \tau_{\ell}})=-(k-1)S_{\ell}\left(\varphi\right)(\overline{\tau_{1}\dots \tau_{\ell}}),
\]
i.e. 
\[
S_{\ell}\left(\psi-\varphi\right)(\overline{\tau_{1}\dots \tau_{\ell}})=0.
\]

We have thus shown that if $\tau_{1}\dots \tau_{\ell}$ contains at least two different letters, then 
\[
S_{\ell}\left(\psi-k\varphi\right)(\overline{\tau_{1}\dots \tau_{\ell}})\neq0\implies S_{\ell}\left(\psi-\varphi\right)(\overline{\tau_{1}\dots \tau_{\ell}})=0.
\]
Therefore, 
\begin{equation}
\label{eq:1or3}
\frac{S_{\ell}\psi\left(\overline{\tau_{1}\dots \tau_{\ell}}\right)}{S_{\ell}\varphi\left(\overline{\tau_{1}\dots \tau_{\ell}}\right)}\in\left\{ 1,k\right\} .
\end{equation}
It remains to consider the fixed points $\overline{\tau_1}$:
\[
\frac{S_{1}\psi\left(\overline{\tau_1}\right)}{S_{1}\varphi\left(\overline{\tau_1}\right)}=\frac{\psi\left(\overline{\tau_1}\right)}{\varphi\left(\overline{\tau_1}\right)}.
\]
Hence, the set
\[
A:=\left\{\frac{S_{\ell}\psi\left(\overline{\tau_{1}\dots \tau_{\ell}}\right)}{S_{\ell}\varphi\left(\overline{\tau_{1}\dots \tau_{\ell}}\right)},\quad \ell \ge 1,\quad \tau_1\dots \tau_\ell \in I \right\}
\]
contains only finitely many values, at least one of which is 1 or $k$ in view of \eqref{eq:1or3}. Now, Lemma \ref{lem:cohom-1} implies
that $\psi$ is cohomologous to $(\dim_{H}\Lambda)\varphi$ and  $\dim_{H}\Lambda=1$.
Hence, $\varphi$ is cohomologous to $\psi$, which is a contradiction. 
\end{proof}

\begin{proof}[Proof of Theorem \ref{thm:main}]
We show that, for any $\alpha>1$, if there {exist $C>0$, $h>0$ and} a polynomial $P$ of degree $n<\alpha$ such that
\begin{equation} \label{eq:polynomial-Holder}
|F(t)-P(t-t_0)|\leq C|t-t_0|^\alpha \qquad \mbox{for all {$t\in(t_0-h,t_0+h)$}},
\end{equation}
then $P$ is constant: $P(t)=F(t_0)$ for every $t\in\RR$. This clearly implies that $\alpha_F(t_0)=\tilde{\alpha}_F(t_0)$.

Suppose, to the contrary, that $P$ has degree $1\leq n<\alpha$ and satisfies \eqref{eq:polynomial-Holder}. We can write
\[
P(t)=a_0+a_1 t+a_2 t^2+\dots +a_n t^n,
\]
where $a_0=P(0)=F(t_0)$. We will show by induction that $a_1=a_2=\dots=a_n=0$. First, write
\[
\frac{F(t)-F(t_0)}{t-t_0}=\frac{F(t)-P(t-t_0)}{t-t_0}+\frac{P(t-t_0)-F(t_0)}{t-t_0},
\]
and observe by \eqref{eq:polynomial-Holder} that
\[
\left|\frac{F(t)-P(t-t_0)}{t-t_0}\right|\leq C|t-t_0|^{\alpha-1} \to 0 \qquad \mbox{as $t\to t_0$}.
\]
Hence,
\[
F'(t_0)=\lim_{t\to t_0}\frac{F(t)-F(t_0)}{t-t_0}=\lim_{t\to t_0}\frac{P(t-t_0)-P(0)}{t-t_0}=P'(0)=a_1.
\]
By Proposition \ref{prop:no-positive-derivative}, it follows that $a_1=0$, establishing the basis for the induction.

Next, take $2\leq j\leq n$, and assume we have already shown that $a_1=\dots=a_{j-1}=0$. Thus,
\[
P(t)=F(t_0)+a_j t^j+\dots+a_n t^n.
\]
Since $j\leq n<\alpha$, we can apply the same idea as above to conclude that
\[
\lim_{t\to t_0}\frac{F(t)-F(t_0)}{(t-t_0)^j}=\lim_{t\to t_0}\frac{P(t-t_0)-P(0)}{(t-t_0)^j}=\lim_{t\to t_0}\sum_{i=j}^n \frac{a_i(t-t_0)^i}{(t-t_0)^j}=a_j,
\]
and therefore, $a_j=0$ by Proposition \ref{prop:no-positive-derivative}.

By induction, we have established the claim, so $P(t)\equiv F(t_0)$, as desired.
\end{proof}

\section{The pointwise H\"older spectrum of $F_\mu$} \label{sec:Holder-spectrum}

Recall from the Introduction that $E_F(\alpha)=\{t\in\RR:\alpha_F(t)=\alpha\}$.
As a consequence of Theorem \ref{thm:main}, we obtain the following result, in which $\alpha_-$ and $\alpha_+$ are given by \eqref{spectrum-endpoints}, and $\beta^*(\alpha)$ is defined by \eqref{eq:Legendre-transform}.

\begin{corollary} \label{cor:Holder-spectrum}
With the notation of Theorem \ref{thm:main}, we have that
\[
\dim_H E_F(\alpha)=\beta^*(\alpha) \qquad \forall\,\alpha\in[\alpha_-,\alpha_+],
\]
whereas $E_F(\alpha)=\emptyset$ for $\alpha\not\in[\alpha_-,\alpha_+]$.
\end{corollary}

\begin{proof}
By \eqref{eq:hoelder exponent as liminf} and Theorem \ref{thm:main}, it follows that
\[
E_F(\alpha)=\EE_*(\alpha,\mu_\psi).
\]
As shown in \cite{JS-2020}, $\dim_H \EE_*(\alpha,\mu_\psi)=\beta^*(\alpha)$ for $\alpha\in[\alpha_-,\alpha_+]$, and $\EE_*(\alpha,\mu_\psi)=\emptyset$ otherwise.
\end{proof}

\begin{remark}
{\rm
The situation is different for the packing spectrum. We illustrate this for the case when $\mu$ is a self-similar measure. Suppose the maps $\{\phi_i\}_{i\in I}$ are similarities with contraction ratios $\{r_i\}$, and $\mathbf{p}=(p_1,\dots,p_m)$ is a probability vector. Put $\psi(\omega)=\log p_{\omega_1}$ for $\omega\in I^\NN$. Likewise we obtain $\varphi(\omega)=\log r_{\omega_1}$. Then $\mu_\psi$ is just the self-similar measure satisfying
\[
\mu=\sum_{i\in I}p_i \mu\circ \phi_i^{-1},
\]
and the pressure equation \eqref{eq:pressure equation} becomes
\[
\sum_{i\in I} r_i^{\beta(q)}p_i^q=1.
\]
Let $A(x;\mu)$ denote the set of accumulation points of $\log \mu(B(x,r))/\log r$ as $r\to 0$. It was shown by Baek, Olsen and Snigireva \cite{BOS} under the strong separation condition, and later by Zhou and Chen \cite{Zhou-Chen} under the open set condition, that for any closed interval $I$, the set
\[
\Delta(I):=\{x\in\RR: A(x;\mu)=I\}
\]
has packing dimension
\[
\dim_P \Delta(I)=\max_{\alpha\in I} \beta^*(\alpha).
\]
Let $\alpha_0$ be the unique value of $\alpha$ maximizing $\beta^*(\alpha)$. Take first $\alpha<\alpha_0$, and put $I=[\alpha,\alpha_0]$. Since $\EE_*(\alpha)\supset \Delta(I)$, it follows that 
\[
\dim_P E_F(\alpha)=\dim_P \EE_*(\alpha;\mu)\geq \dim_P \Delta(I)=\beta^*(\alpha_0)=\dim_P \Lambda.
\]
The reverse inequality is trivial since $E_F(\alpha)\subset \Lambda$. Hence, $\dim_P E_F(\alpha)=\beta^*(\alpha_0)$.

For $\alpha\geq\alpha_0$, we trivially have $\dim_P \EE_*(\alpha;\mu)\geq \dim_P\EE(\alpha;\mu)=\beta^*(\alpha)$ since $\EE_*(\alpha;\mu)\supset\EE(\alpha,\mu)$. But we also have $\dim_P \EE_*(\alpha;\mu)\leq \beta^*(\alpha)$. This follows directly from a careful examination of the proof that $\dim_P \EE(\alpha;\mu)=\beta^*(\alpha)$; see, for instance, \cite[p.~501]{Patzschke}. 
We thus conclude that
\begin{equation} \label{eq:packing-spectrum}
\dim_P E_F(\alpha)=\begin{cases}
\beta^*(\alpha_0) & \mbox{if $\alpha_-\leq\alpha\leq\alpha_0$},\\
\beta^*(\alpha) & \mbox{if $\alpha_0\leq\alpha\leq \alpha_+$}.
\end{cases}
\end{equation}
Note that, on the left half of the multifractal spectrum, the Hausdorff dimension does not differentiate between the sets $\EE(\alpha;\mu)$ and $\EE_*(\alpha;\mu)$, but the packing dimension does. 

}
\end{remark}

\section{Conjugacy maps between expanding piecewise $\mathcal{C}^{1+\epsilon}$ maps} \label{sec:conjugacy-maps}

We apply our results to conjugacy maps 
between expanding piecewise $\mathcal{C}^{1+\epsilon}$ interval maps, 
extending  \cite[Corollary 2.3]{JS-2020} in two ways. Firstly, we also analyze the pointwise  H\"older exponent $\alpha$ of the conjugacy map, whereas in   \cite{JS-2020} only $\tilde \alpha$ was considered. Secondly, we allow interval maps with orientation-reversing branches. 

Let $f$ be an expanding
piecewise $\mathcal{C}^{1+\epsilon}$ interval map with $s\ge2$ full
branches, i.e., there exist closed intervals $J_{1},\dots,J_{s}\subset[0,1]$
with non-empty, pairwise disjoint interiors such that $f|_{J_{i}}$
has a $\mathcal{C}^{1+\epsilon}$ extension to a neighbourhood of
$J_{i}$ satisfying $|f'|_{J_{i}}|>1$ and $f(J_{i})=[0,1]$, for $1\le i\le s$.
We will always assume that $\sup J_{i}\le\inf J_{j}$ if $i<j$.
The repeller of $f$ is denoted by $\Lambda$ and the restriction
$f|_{\Lambda}:\Lambda\rightarrow\Lambda$ is semi conjugate to the shift
$\sigma:I^{\N}\rightarrow I^{\N}$ with alphabet $I:=\left\{ 1,\dots,s\right\} $.
The semi conjugacy is given by the coding map $\pi_{f}:I^{\N}\rightarrow\Lambda$
of the conformal iterated function systems $\Phi_{f}$, which is given
by the contracting inverse branches $(f|_{J_{i}})^{-1}:[0,1]\rightarrow [0,1]$,
$1\le i\le s$. 
Note that each branch is either orientation-preserving, or orientation-reversing.  We denote by $\varphi_{f}:I^{\N}\rightarrow\R$
the geometric potential of $\Phi_{f}$ given by $\varphi_{f}:=-\log |f'\circ\pi_{f}|$. 

We say that two piecewise $\mathcal{C}^{1+\epsilon}$ interval maps $f$ and $g$ have the {\em same signature} if they share the same alphabet $I$ and satisfy 
$$\{i\in I\mid f_i \text{ is orientation-preserving }\}=\{i\in I\mid g_i \text{ is orientation-preserving }\}.$$ 
If $f$ and $g$ have the same signature and $\Lambda_g=[0,1]$, then 
$\pi_g(\pi_f^{-1}\{x\})$ is  a singleton for each $x\in \Lambda_f$. We can thus define the conjugacy 
map 
\[\Theta:\Lambda_f\rightarrow[0,1], \quad \{\Theta(x)\}=\pi_g(\pi_f^{-1}\{x\}).
\] 
Slightly abusing notation, we write $\Theta=\pi_g\circ\pi_f^{-1}$. Note that we have 
 $\Theta\circ f=g\circ\Theta$, which implies for each $x\in\Lambda_f$,  
 \begin{equation}
\pi_{g}^{-1}\left((-\infty,\Theta(x)]\right)=\pi_{f}^{-1}\left((-\infty,x]\right).\label{eq:sets}
\end{equation}
The conjugacy $\Theta$ has a unique continuous extension to a function on $\R$ which is locally constant on $\R \setminus \Lambda_f$.  We denote this function again by $\Theta$ and refer to it as the conjugacy map between $f$ and $g$. Note that \eqref{eq:sets} continues to hold for the conjugacy map between $f$ and $g$.  

Since $\Lambda_g=[0,1]$ we have that the Lebesgue measure $\lambda$ on $[0,1]$ satisfies $\lambda=\mu_{\varphi_{g}}\circ\pi_{g}^{-1}$ where $\mu_{\varphi_{g}}$ is a Gibbs measure for $\varphi_g$. Hence, by  \eqref{eq:sets}, the conjugacy map $\Theta$ between $f$ and $g$    coincides with the distribution function of $\mu_{\varphi_{g}}\circ\pi_{f}^{-1}$. Our main theorem thus implies the following. 

\begin{corollary} \label{cor}
Let $f$ and $g$ be two expanding piecewise $\mathcal{C}^{1+\epsilon}$
interval maps 
with full branches 
and with the same signature, {and assume $\Lambda_g=[0,1]$.} Suppose that $\varphi_f$ is not cohomologous to $\varphi_g$.  Then, for the conjugacy map $\Theta:\RR\rightarrow [0,1]$ between  
between $f$ and $g$ we have 
\[
\dim_{H}\left\{ t\in\RR\mid\alpha_\Theta(t)=\alpha\right\}=\dim_{H}\left\{ t\in\RR\mid\tilde{\alpha}_\Theta(t)=\alpha\right\} =\beta^{*}(\alpha),\quad \alpha \in\left[\alpha_{-},\alpha_{+}\right],
\]
where $\beta:\R\rightarrow\R$ is defined implicitly by $\mathcal{P}(\beta(q)\varphi_{f}+q\varphi_{g})=0$, and $\alpha_-<\alpha_+$.
\end{corollary}

\begin{remark} Let $f$ and $g$ be two expanding piecewise $\mathcal{C}^{1+\epsilon}$ interval maps 
with full branches. Suppose in addition that all branches of $f$ and $g$ are monotonically increasing and that $\Lambda_f = \Lambda_g=[0,1]$.  Then, it follows from \cite[Proof of Theorem 1.2]{JKPS-2009} that the conjugacy map $\Theta:[0,1]\rightarrow [0,1]$ between  $f$ and $g$ is a $\mathcal{C}^{1+\epsilon}$-diffeomorphism if and only if $\varphi_f$ is cohomologous to $\varphi_g$. 
\end{remark}  

\begin{example} \label{ex:conjugacy}
{\rm
Let
\[
f(x):=\begin{cases} 3x & \mbox{if $0\leq x\leq 1/3$},\\
3x-2 & \mbox{if $2/3\leq x\leq 1$}, 
\end{cases}
\qquad
g(x):=\begin{cases} 2x & \mbox{if $0\leq x<1/2$},\\
2x-1 & \mbox{if $1/2\leq x\leq 1$}.
\end{cases}
\]
Then the repeller $\Lambda$ of $f$ is the ternary Cantor set and the conjugacy $\Theta$ between $f$ and $g$ is the ternary Cantor function.
Suppose we invert the second branch of both $f$ and $g$; that is, we consider instead the functions
\[
\tilde{f}(x):=\begin{cases} 3x & \mbox{if $0\leq x\leq 1/3$},\\
4-3x & \mbox{if $2/3\leq x\leq 1$}, 
\end{cases}
\qquad
\tilde{g}(x):=\begin{cases} 2x & \mbox{if $0\leq x<1/2$},\\
3-2x & \mbox{if $1/2\leq x\leq 1$}.
\end{cases}
\]
Note that $\tilde{f}$ and $\tilde{g}$ again have the same signature.
The repeller of $\tilde{f}$ is still the ternary Cantor set, and the conjugacy $\tilde{\Theta}$ between $\tilde{f}$ and $\tilde{g}$ is still the ternary Cantor function. In other words, $\tilde{\Theta}:=\pi_{\tilde g}\circ\pi_{\tilde{f}}^{-1}=\Theta$. This follows because, if the coding of $x\in\Lambda$ under $f$ is $\omega:=(\omega_n)_n:=\pi_f^{-1}(x)$, then the coding of $x$ under $\tilde{f}$ is $\tilde{\omega}=(\tilde{\omega}_n)_n=\pi_{\tilde f}^{-1}(x)$, where $\tilde{\omega}_1=\omega_1$ and successively, for $n=2,3,\dots$,
\[
\tilde{\omega}_n:=\begin{cases}
\omega_n & \mbox{if $\#\{i<n: \tilde{\omega}_i=2\}$ is even},\\
3-\omega_n & \mbox{otherwise},
\end{cases}
\]
as is easy to see. We then have both $\pi_f(\omega)=\pi_{\tilde f}(\tilde{\omega})=x$ and $\pi_g(\omega)=\pi_{\tilde g}(\tilde{\omega})$, so that $\tilde{\Theta}(x)=\pi_{\tilde g}(\pi_{\tilde f}^{-1}(x)) =\pi_g(\pi_f(x))=\Theta(x)$. Observe that the spectrum in this example is trivial, i.e.  $\alpha_-=\alpha_+=\log 2/\log 3$.

}
\end{example}

\begin{example}
{\rm
 More generally, let $m,n\in \N$ with $m,n\ge 3$. Let
\[
f(x):=\begin{cases} mx & \mbox{if $0\leq x\leq 1/m$},\\
nx-(n-1) & \mbox{if $(n-1)/n\leq x\leq 1$}, 
\end{cases}
\]
and let $g(x)$ again be as in Example \ref{ex:conjugacy}.
If $m=n$, the situation is much as it was in the last example. But if $m\neq n$, the conjugacy $\Theta$ changes when we reverse the second branch of both $f$ and $g$, because the repeller $\Lambda$ changes. Nonetheless, the new conjugacy $\tilde{\Theta}$ has the same H\"older spectrum as $\Theta$. The range of the spectrum is 
\[
[\alpha_-,\alpha_+]=\left[\min\left\{\frac{\log 2}{\log m},\frac{\log 2}{\log n}\right\},\max\left\{\frac{\log 2}{\log m},\frac{\log 2}{\log n}\right\}\right],
\]
and the spectrum is given by Corollary \ref{cor}.
}
\end{example}

\section*{Acknowledgments}
{The authors wish to thank two anonymous referees for their careful reading of the paper.}
P.A. was partially supported by Simons Foundation grant \#709869.
J.J. was partially supported by the JSPS KAKENHI 24K06777.


\small

\begin{thebibliography}{12}

\bibitem{Allaart-2018}
P.~C. Allaart, Differentiability and {H}\"{o}lder spectra of a class of
  self-affine functions, {\em Adv. Math.} \textbf{328} (2018), 1--39. 

\bibitem{Allaart-2020}
{P. C. Allaart}, The pointwise H\"older spectrum of general self-affine functions on an interval. {\em J. Math. Anal. Appl.} {\bf 488} (2020), no. 2, Article 124096, 35 pp.

\bibitem{BOS}
{I. S. Baek}, {L. Olsen} and {N. Snigireva}, Divergence points of self-similar measures and packing dimension. 
{\em Adv. Math.} {\bf 214} (2007), no. 1, 267--287.

\bibitem{BKK}
B.~B\'{a}r\'{a}ny, G.~Kiss, and I.~Kolossv\'{a}ry, Pointwise regularity
  of parameterized affine zipper fractal curves, {\em Nonlinearity} \textbf{31} (2018), no.~5, 1705--1733. 


\bibitem{BenSlimane}
{M. Ben Slimane}, Multifractal formalism for selfsimilar functions expanded in singular basis. {\em Appl. Comput. Harmon. Anal.} {\bf 11} (2001), 387--419.

\bibitem{bowen-equilibrium}
R.~Bowen, \emph{Equilibrium states and the ergodic theory of {A}nosov
  diffeomorphisms}, Springer-Verlag, Berlin, 1975, Lecture Notes in
  Mathematics, Vol. 470. 


\bibitem{Falconer-2004}
K.~Falconer, One-sided multifractal analysis and points of
  non-differentiability of devil's staircases, {\em Math. Proc. Cambridge Philos.  Soc.} \textbf{136} (2004), no.~1, 167--174. 

\bibitem{Frisch-Parisi}
U.~Frisch and G.~Parisi, On the singularity structure of fully developed
  turbulence. {\em Turbulence and predictability in geophysical fluid dynamics and climate dynamics} (North Holland Amsterdam), 1985, pp.~84--88.




\bibitem{JS-2020}
{J. Jaerisch} and {H. Sumi}, Multifractal Formalism for generalised local dimension spectra of Gibbs measures on the real line. {\em J. Math. Anal. Appl.} {\bf 491} (2020), no. 2, Article 124246, 9 pp.

\bibitem{Jaffard1}
{S. Jaffard}, Multifractal formalism for functions part I: results valid for all functions. {\em SIAM J. Math. Anal.} {\bf 28} (1997), no. 4, 944--970.

\bibitem{Jaffard2}
{S. Jaffard}, Multifractal formalism for functions part II: self-similar functions.  {\em SIAM J. Math. Anal.} {\bf 28} (1997), no. 4, 971--998.

\bibitem{JafMan}
{S. Jaffard} and {B. B. Mandelbrot}, Local regularity of nonsmooth wavelet expansions and application to the Polya function. {\em Adv. Math} {\bf 120} (1996), 265--282.

\bibitem{JKPS-2009}
T.~Jordan, M.~Kesseb{\"o}hmer, M.~Pollicott, and B.~O. Stratmann, Sets of nondifferentiability for conjugacies between expanding interval maps, {\em Fund.  Math.} \textbf{206} (2009), 161--183. 


\bibitem{KS-2009}
M.~Kesseb{{\"o}}hmer and B.~O. Stratmann, H{\"o}lder-differentiability of 
{G}ibbs distribution functions, {\em Math. Proc. Cambridge Philos. Soc.}
  \textbf{147} (2009), no.~2, 489--503. 


\bibitem{Mandelbrot}
B.~B. Mandelbrot, Intermittent turbulence in self-similar cascades:
  divergence of high moments and dimension of the carrier, {\em J. Fluid Mechanics Digital Archive} \textbf{62} (1974), no.~2, 331--358.

\bibitem{Mauldin-Urbanski}
R.~D. Mauldin and M.~Urba{\'n}ski, Dimensions and measures in infinite
  iterated function systems, {\em Proc. London Math. Soc.} (3) \textbf{73} (1996),  no.~1, 105--154. 

\bibitem{MauUrb03} 
R. D. Mauldin and M. Urba\'nski.
{\it Graph directed Markov systems: Geometry and Dynamics of Limit Sets.} Cambridge Tracts in Mathematics
 {\bf  148} Cambridge University Press (2003)

\bibitem{Otani}
A.~Otani, Fractal dimensions of graph of {W}eierstrass-type function and local {H}\"{o}lder exponent spectra, {\em Nonlinearity} \textbf{31} (2018), no.~1, 263--292. 

\bibitem{Patzschke}
N.~Patzschke, Self-conformal multifractal measures, {\em Adv. in Appl. Math.}  \textbf{19} (1997), no.~4, 486--513. 


\bibitem{pesin-dimension-theory}
Y.~B. Pesin, \emph{Dimension theory in dynamical systems}, Chicago Lectures in  Mathematics, University of Chicago Press, Chicago, IL, 1997, Contemporary views and applications. 

\bibitem{Pesin-Weiss-97a}
Y.~Pesin and H.~Weiss, A multifractal analysis of equilibrium measures
  for conformal expanding maps and {M}oran-like geometric constructions, {\em J. Statist. Phys.} \textbf{86} (1997), no.~1-2, 233--275. 



\bibitem{schmeling1999}
J. Schmeling, On the completeness of multifractal spectra. {\em Ergodic Theory Dynam. Systems}. \textbf{19}, 1595-1616 (1999)

\bibitem{Seuret}
{S. Seuret}, On multifractality and time subordination for continuous functions. {\em Adv. Math.} {\bf 220} (2009), no. 3, 936--963.

\bibitem{Troscheit}
S.~Troscheit, H\"{o}lder differentiability of self-conformal devil's
  staircases, {\em Math. Proc. Cambridge Philos. Soc.} \textbf{156} (2014), no.~2,
  295--311. 

\bibitem{Zhou-Chen}
{X. Zhou} and {E. Chen}, Packing dimensions of the divergence points of self-similar measures with OSC.
{\em Monatsh. Math.} {\bf 172} (2013), no. 2, 233--246.

\end{thebibliography}
\end{document}